\documentclass{amsart}

\usepackage[T2A,T1]{fontenc}
\usepackage[utf8]{inputenc}
\usepackage[russian,USenglish]{babel}
\usepackage{amsmath,amssymb,amsthm,upgreek}
\usepackage{colortbl,color,xcolor}
\usepackage{graphicx,tikz,pgfplots,pgfplotstable}
\usepackage{enumitem}
\pgfplotsset{compat=1.6}
\usepackage{subcaption,float}
\usepackage[bookmarks=true,bookmarksnumbered,plainpages=false,linktocpage,colorlinks=true,citecolor=green!80!black,linkcolor=red!70!black,filecolor=magenta,urlcolor=magenta,hidelinks,breaklinks,pdfauthor={Frank Filbir, Ralf Hielscher, Thomas Jahn, Tino Ullrich},pdftitle={Marcinkiewicz--Zygmund inequalities for scattered and random data on the sphere},unicode=true]{hyperref}
\usepackage[nameinlink,capitalize,noabbrev]{cleveref}
\usepackage{cite}

\usepackage{ifthen}

\makeatletter
\@namedef{subjclassname@2010}{\textup{2020} Mathematics Subject Classification}
\makeatother

\newcommand{\rus}[1]{\selectlanguage{russian}{\fontfamily{cmr}\fontsize{10pt}{12pt}\selectfont #1}\selectlanguage{USenglish}}
\newcommand{\bibrus}[1]{\selectlanguage{russian}{\fontfamily{cmr}\fontsize{8pt}{10pt}\selectfont #1}\selectlanguage{USenglish}}

\newcommand{\enquote}[1]{``#1''}    
\newcommand{\ii}{\mathrm{i}}
\newcommand{\ee}{\mathrm{e}}
\newcommand{\dd}{\mathrm{d}}
\newcommand{\NN}{\mathbb{N}}
\newcommand{\RR}{\mathbb{R}}
\newcommand{\ZZ}{\mathbb{Z}}
\newcommand{\CC}{\mathbb{C}}
\newcommand{\abs}[1]{\left\lvert#1\right\rvert}
\newcommand{\mnorm}[1]{\left\lVert#1\right\rVert}
\newcommand{\setn}[1]{\left\{#1\right\}}
\newcommand{\setcond}[2]{\left\{#1 \::\: #2\right\}}
\newcommand{\defeq}{\mathrel{\mathop:}=}
\newcommand{\lr}[1]{\!\left(#1\right)}
\newcommand{\skpr}[2]{\left\langle#1 \,\middle\vert\, #2\right\rangle}
\newcommand{\norel}{\mathrel{\phantom{=}}}
\newcommand{\eqdef}{=\mathrel{\mathop:}}
\newcommand{\floor}[1]{\left\lfloor #1\right\rfloor}

\newcommand{\smpl}{\Xi}
\newcommand{\sph}{\mathbb{S}}
\newcommand{\prt}{\mathcal{Z}}
\newcommand{\Z}{Z}
\newcommand{\TT}{\mathbb{T}}
\newcommand{\PP}{\mathbb{P}}

\DeclareMathOperator*{\esssup}{ess\,sup}

\DeclareMathOperator{\inte}{int}

\let \eps \varepsilon
\let \piup \uppi

\theoremstyle{plain} 
\newtheorem{Satz}{Theorem}[section]\newtheorem{Kor}[Satz]{Corollary}
\newtheorem{Lem}[Satz]{Lemma}
\newtheorem{Prop}[Satz]{Proposition}
\newtheorem*{thmmatrix}{Theorem~3.2}
\newtheorem*{thmdeterm}{Theorem~4.1}
\newtheorem*{thmcoupon}{Theorem~5.2}

\theoremstyle{definition}

\newtheorem{Bem}[Satz]{Remark}

\crefname{Satz}{Theorem}{Theorems}
\crefname{Prop}{Proposition}{Propositions}
\crefname{Lem}{Lemma}{Lemmas}
\crefname{Bem}{Remark}{Remarks}
\crefname{Kor}{Corollary}{Corollaries}
\crefname{Def}{Definition}{Definitions}

\numberwithin{equation}{section}

\makeatletter
\renewcommand*{\eqref}[1]{\hyperref[{#1}]{\textup{\tagform@{\ref*{#1}}}}}
\makeatother

\newcommand\pgfmathsinandcos[3]{%
  \pgfmathsetmacro#1{sin(#3)}%
  \pgfmathsetmacro#2{cos(#3)}%
}
\newcommand\LongitudePlane[3][current plane]{%
  \pgfmathsinandcos\sinEl\cosEl{#2} 
  \pgfmathsinandcos\sint\cost{#3} 
  \tikzset{#1/.style={cm={\cost,\sint*\sinEl,0,\cosEl,(0,0)}}}
}
\newcommand\LatitudePlane[3][current plane]{%
  \pgfmathsinandcos\sinEl\cosEl{#2} 
  \pgfmathsinandcos\sint\cost{#3} 
  \pgfmathsetmacro\yshift{\cosEl*\sint}
  \tikzset{#1/.style={cm={\cost,0,0,\cost*\sinEl,(0,\yshift)}}} %
}
\newcommand\DrawLongitudeCircle[2][1]{
  \LongitudePlane{\angEl}{#2}
  \tikzset{current plane/.prefix style={scale=#1}}
  \pgfmathsetmacro\angVis{atan(sin(#2)*cos(\angEl)/sin(\angEl))} %
  \draw[current plane] (\angVis:1) arc (\angVis:\angVis+180:1);
  \draw[current plane,dashed] (\angVis-180:1) arc (\angVis-180:\angVis:1);
}
\newcommand\LatitudeCircle[2][1]{
  \LatitudePlane{\angEl}{#2}
  \tikzset{current plane/.prefix style={scale=#1}}
  \pgfmathsetmacro\sinVis{sin(#2)/cos(#2)*sin(\angEl)/cos(\angEl)}
  \pgfmathsetmacro\angVis{asin(min(1,max(\sinVis,-1)))}
  \path[current plane] (\angVis:1) arc (\angVis:-\angVis-180:1);
  \path[current plane,dashed] (180-\angVis:1) arc (180-\angVis:\angVis:1);
}

\makeatletter
\pgfdeclareradialshading[tikz@ball]{ball}{\pgfqpoint{-10bp}{10bp}}{%
 color(0bp)=(tikz@ball!30!white);
 color(9bp)=(tikz@ball!75!white);
 color(18bp)=(tikz@ball!90!black);
 color(25bp)=(tikz@ball!70!black);
 color(50bp)=(black)}
\makeatother

\begin{document}
\allowdisplaybreaks

\title[MZ inequalities for scattered and random data on the sphere]{Marcinkiewicz--Zygmund inequalities for scattered and random data on the $q$-sphere}


\author{Frank Filbir}
\address[F.~Filbir]{Mathematical Imaging and Data Analysis, Institute of Biological and Medical Imaging, Helmholtz Center Munich, Ingolstädter Landstraße 1, 85764 Neuherberg, Germany}
\email{filbir@helmholtz-muenchen.de}
\thanks{}

\author{Ralf Hielscher}
\address[R.~Hielscher]{Faculty of Mathematics and Computer Science, Technische Universität Bergakademie Freiberg, 09596 Freiberg, Germany}
\email{ralf.hielscher@math.tu-freiberg.de}
\thanks{}

\author{Thomas Jahn}
\address[T.~Jahn]{Mathematical Institute for Machine Learning and Data Science (MIDS), Catholic University of Eichstätt--Ingolstadt (KU), Auf der Schanz 49, 85049 Ingolstadt, Germany}
\email{thomas.jahn@ku.de}
\thanks{}

\author{Tino Ullrich}
\address[T.~Ullrich]{Faculty of Mathematics, Technische Universität Chemnitz, 09107 Chemnitz, Germany}
\email{tino.ullrich@mathematik.tu-chemnitz.de}
\thanks{}



\date{\today}

\begin{abstract}
The recovery of multivariate functions and estimating their integrals from finitely many samples is one of the central tasks in modern approximation theory.
Marcinkiewicz--Zygmund inequalities provide answers to both the recovery and the quadrature aspect.
In this paper, we put ourselves on the $q$-dimensional sphere $\sph^q$, and investigate how well continuous $L_p$-norms of polynomials $f$ of maximum degree $n$ on the sphere $\sph^q$ can be discretized by positively weighted $L_p$-sum of finitely many samples, and discuss the relationship between the offset between the continuous and discrete quantities, the number and distribution of the (deterministic or randomly chosen) sample points $\xi_1,\ldots,\xi_N$ on $\sph^q$, the dimension $q$, and the polynomial degree $n$.
\end{abstract}

\maketitle
\section{Introduction}
A typical problem in science is to develop a model for a hidden process from observational data.
More precisely, we are given a set of measurements $\{(x_1,f_1),\ldots,$ $(x_N,f_N)\}$, where we assume that the set of sampling nodes $\smpl=\setn{x_1,\ldots,x_N}$ is a finite subset of a compact metric measure space $\mathbb{X}$ with measure $\mu$ and metric $\rho$.
The vector of sampling values $S(f) = (f_1,\ldots,f_N)$ has real or complex components.
It is usually assumed that the data generating process can be described by a complex-valued function $f$ defined on $\mathbb{X}$, viz.\ $f(x_j)=f_j$, or at least $f(x_j)\approx f_j$.
In order to develop a mathematical method to approximate the function $f$ from its samples it is necessary to make suitable assumptions regarding the nature of $f$.
That is, we assume that $f$ belongs to a (smoothness) function space at least embedded into $C(\mathbb{X})$ in order to make function evaluation available. 
The question of which function space is suitable is not primarily a mathematical problem but depends more on the specific application.
The mathematical problem is to determine an approximation $P \in \Pi_M$ to $f$ from the given data with a certain accuracy, and to give error bounds for this approximation. A common strategy is to \enquote{project} $f$ onto this finite-dimensional subspace $\Pi_M$ spanned by the first $M$ elements of an orthonormal basis $\setcond{\phi_n}{n\in\NN}$ of $L_2(\mathbb{X},\mu)$ by only using the above mentioned discrete information.
This is usually done via interpreting the discrete information as \enquote{noisy} samples of a model function $\tilde{f} \in \Pi_M$ and using a least squares approach for the recovery of the coefficients as soon as the sampling operator $\mathcal{S}:\Pi_M\to\RR^N$, $\mathcal{S}(P)=(P(x_1),\ldots,P(x_N))$ is bounded and boundedly invertible on its range, i.e.,
\begin{equation}\label{eq:In2}
c_1\mnorm{P}_p\leq\mnorm{\mathcal{S}(P)}_{\omega,p,\RR^N}\leq c_2\mnorm{P}_p
\end{equation}
for all $P\in\Pi_M$, where $0<c_1\leq c_2,\ 1\leq p\leq\infty$, and $\mnorm{\cdot}_ {\omega,p,\RR^N}$ is a certain weighted discrete $p$-norm on $\RR^N$.


Broadly speaking, inequality \eqref{eq:In2} is about the discretization of $L_p$-norms on $\Pi_n^q$, i.e., polynomials on $\sph^q$ with maximum degree $n$, using point samples.
The present paper is concerned with such inequalities on the $q$-dimensional unit sphere $\mathbb{X} = \sph^q$ in $\RR^{q+1}$.
Our results cover variants of this \emph{Marcinkiewicz--Zygmund inequality} for both scattered (i.e., deterministically given) sampling points and randomly chosen ones.

For scattered data, we establish the following $L_p$-result by applying Riesz--Thorin interpolation to the boundary cases $p=1$ and $p=\infty$. Here, $L_p$-norms are computed with respect to the surface area measure $\mu_q$ of $\sph^q$, and the weights in the discretized norm are given by the surface areas of the patches $Z_\xi$ of a partition of $\sph^q$, each of which belongs to a sampling point $\xi\in\smpl$.
The geometry of the partition enters through the partition norm $\mnorm{\prt}$, i.e., the maximum geodesic diameter of its patches.
\begin{thmdeterm}
Let $\eta\in \mathopen]0,1\mathclose[$, and let $(\smpl,\prt)$ be a compatible pair consisting of a finite set $\smpl\subseteq \sph^q$ and a partition $\prt$ of $\sph^q$.
Assume that
\begin{equation*}
5C_q(n+q^2)\mnorm{\prt}\leq \eta
\end{equation*}
with $C_q\defeq 2(3+3^{q/2}\,\piup)$.
Then, for all $p\in[1,\infty]$ and every $P\in\Pi_n^q$, we have
\begin{equation*}
(1-\eta)\mnorm{P}_{\mu_q,p}\leq\lr{\sum_{\xi\in \smpl}\mu_q(Z_\xi)\abs{P(\xi)}^p}^{\frac{1}{p}}\leq (1+\eta)\mnorm{P}_{\mu,p}.
\end{equation*}
\end{thmdeterm}

Inequalities of this type have been considered by Marcinkiewicz and Zygmund in their seminal paper \cite{MarcinkiewiczZy1937} in relation to interpolation problems for functions defined on the torus $\mathbb{X}=\sph^1$ (resp.\ $2\piup$-periodic functions) on equidistant nodes $x_j=j/N$. 
More precisely, the authors of \cite{MarcinkiewiczZy1937} proved that for every trigonometric polynomial $P$ of maximal degree $n$ and every $1\leq p\leq\infty$ the following chain of inequalities hold
\begin{equation}\label{eq:In0}
\lr{1-\eps} \lr{\frac{1}{N}\sum_{j=1}^N\abs{P(x_j)}^p}^{\frac{1}{p}}\leq \mnorm{P}_p
\leq \lr{1+\eps}\lr{\frac{1}{N}\sum_{j=1}^N\abs{P(x_j)}^p}^{\frac{1}{p}},
\end{equation}
provided that the number $N$ of sampling points is strictly greater than $(1+\frac{1}{\eps})2n$ for some $\eps>0$.
In this sense, \cref{thm:deterministic-mz} can be considered a generalization of \eqref{eq:In0} on the $q$-sphere, which provides exact constants. 

For sampling points drawn randomly, i.e., independently and identically distributed according to the normalized surface area measure $\sigma_q$, the following $L_2$-version of the Marcinkiewicz--Zygmund inequality will be derived using singular value estimates for random matrices \cite{Tropp2012}.
\begin{thmmatrix}
Let $\eta,\eps\in \mathopen]0,1\mathclose[$.
Suppose $\xi_1,\ldots,\xi_N\in \sph^q$ are drawn i.i.d.\ according to $\sigma_q$.
If 
\begin{equation*}
N>\log\lr{\frac{2\dim(\Pi_n^q)}{\eps}}\frac{3\dim(\Pi_n^q)}{\eta^2},
\end{equation*}
then with probability exceeding $1-\eps$ with respect to the product measure $\PP=\sigma_q^{\otimes N}$, we have
\begin{equation*}
(1-\eta)\mnorm{P}_{\sigma_q,2}^2\leq \frac{1}{N}\sum_{j=1}^N \abs{P(\xi_j)}^2\leq (1+\eta)\mnorm{P}_{\sigma_q,2}^2
\end{equation*}
for all $P\in\Pi_n^q$.
\end{thmmatrix}
A combination of our deterministic Marcinkiewicz--Zygmund inequality in \cref{thm:deterministic-mz}, the well-known coupon collector problem from probability theory \cite[p.~36]{Doerr2020}, and partitioning results for $\sph^q$ \cite[Theorem~3.1.3]{Leopardi2007} leads us to the following $L_p$-Marcinkiewicz--Zygmund inequality for sets of random sampling points.
\begin{thmcoupon}
Let $n,q\in\NN$, $\eta,\eps\in \mathopen]0,1\mathclose[$, $\omega_q\defeq \mu_q(\sph^q)$, and $\alpha_q\defeq 8\lr{\frac{\omega_qq}{\omega_{q-1}}}^{\frac{1}{q}}$.
Choose $N\in\NN$ large enough such that
\begin{equation*}
5C_q\alpha_q\lr{\frac{N}{2\log(\frac{N}{\eps})}}^{-\frac{1}{q}}(n+q^2)<\eta.
\end{equation*}
Draw points $\xi_1,\ldots,\xi_N\in\sph^q$ i.i.d.\ according to $\sigma_q$.
Then with probability $\geq 1-\frac{1}{N}$ with respect to the product measure $\PP=\sigma_q^{\otimes N}$, there exists weights $w_1,\ldots,w_N>0$ such that $\sum_{j=1}^N w_j=1$ and 
\begin{equation*}
(1-\eta)\mnorm{P}_{\sigma_q,p}\leq \lr{\sum_{j=1}^N w_j \abs{P(\xi_j)}^p}^\frac{1}{p}\leq (1+\eta)\mnorm{P}_{\sigma_q,p}
\end{equation*}
for all $p\in [1,\infty]$ and all $P\in\Pi_n^q$.
\end{thmcoupon}

The original Marcinkiewicz--Zygmund inequality \eqref{eq:In0} has been generalized in many directions as to univariate and multivariate algebraic polynomials, to non-equidistant, scattered, or random samplings point sets, and to general manifolds.
These generalizations have many applications in various fields in applied mathematics such as interpolation and approximation, quadrature and optimal design, sampling theory, and phase retrieval.
The number of papers dealing with approximation problems on the sphere related to Marcinkiewicz--Zygmund inequalities is too large to present an exhaustive list here, exemplary we mention the papers \cite{BartelKaPoUl2022,BartelScUl2022,BrownDa2005,BuhmannDaNi2021,DaiPrShTeTi2021b,DaiTe2022,MaggioniMh2008,FilbirMh2010,FilbirMh2011,FilbirTh2004,FilbirTh2008,FreemanGh2023,KashinKoLiTe2022,Kosov2021,KriegSo2021,Lubinsky2014,Mhaskar2020,MhaskarNaWa2001,Temlyakov2017,Temlyakov2018}.
An elaborate discussion on the various relationships and a discussion of related work is given by Gröchenig \cite{Groechenig2020} and by Kashin et al. \cite{KashinKoLiTe2022}.

The reason for revisiting the problem of  the Marcinkiewicz--Zygmund inequalities on the unit sphere $\sph^q$ in this paper is at least twofold.
First, classical proofs of the Marcinkiewicz--Zygmund inequalities for cases $p=1$ and $p=\infty$ are based on the Bernstein inequality.
To get the intermediate cases $1<p<\infty$, commonly a Riesz--Thorin interpolation argument has been employed.
However, there is a pitfall in this argument.
The space $\Pi_n^q$ does not contain the simple functions and therefore the use of the Riesz--Thorin interpolation theorem is not justified.
The authors of \cite{FilbirMh2011} found a workaround to this problem in a rather abstract way.
In the paper at hand, we present a more direct solution to this problem by constructing an operator related to the Marcinkiewicz--Zygmund inequalities which is defined on the entire space $L_p$ and for which the Riesz--Thorin argument is justified, see \cref{thm:deterministic-mz}.
Second, we utilize the deterministic Marcinkiewicz--Zygmund inequality to generalize the probabilistic Marcinkiewicz--Zygmund inequality in \cref{thm:tropp-mz} to general $1\le p \le \infty$ in \cref{thm:leopardi-mz}. These are to some extent easier to set up because unlike in the deterministic version, no partition of the sphere is required.

We have organized the paper as follows.
We start by collecting some basic material regarding the analysis on the $q$-dimensional unit sphere in \cref{sec:prelim}.
\cref{sec:random_first} is devoted to a first look to Marcinkiewicz--Zygmund inequalities for sets of random sampling points in the special case of $L_2$.
The entire \cref{sect:scattered} is concerned with the proof of the Marcinkiewicz--Zygmund inequalities for deterministic sets of scattered sampling points for $L_p$, $1\leq p\leq\infty$.
In \cref{chap:final} we consider the case of random point sets again and we show how to derive $L_p$-versions of the desired inequalities for those sampling sets and $1\leq p\leq \infty$.


\section{Preliminaries}\label{sec:prelim}
We start with some notation and basic results on harmonic analysis on the sphere, which can be found, e.g., in \cite{AtkinsonHan2012}.
Let $q\geq 2$ be an arbitrary but fixed integer.
The $q$-dimensional unit sphere $\sph^q$ embedded in $\RR^{q+1}$ is the set  
\begin{equation*}
\sph^q=\setcond{x\in\RR^{q+1}}{\abs{x}_2=1},
\end{equation*}
where $\abs{x}_2$ denotes the Euclidean norm of $x\in\RR^{q+1}$. For the inner product of two vectors $x,y\in\RR^{q+1}$ we write 
$x\cdot y$.
The geodesic distance on $\sph^q$ is given by $d(x,y)=\arccos(x\cdot y)$.
It defines a metric on $\sph^q$. 
The surface measure on $\sph^q$ will be denoted by $\mu_q$ and we assume that
\begin{equation*}
\mu_q(\sph^q)=\frac{2\piup^{\frac{q+1}{2}}}{\Gamma(\frac{q+1}{2})}\eqdef\omega_q.
\end{equation*}
The spaces $L_p(\sph^q)\defeq L_p(\sph^q,\mu_q)$ are defined as usual.
The inner product on the Hilbert space $L_2(\sph^q)$ is given by 
\begin{equation*}
\skpr{f}{g}=\int_{\sph^q} f(x)\overline{g(x)}\dd \mu_q(x). 
\end{equation*}
Recall that using polar coordinates the $k$th component of the vector $x\in\sph^q$ satisfies 
\begin{equation*}
x_k=\begin{cases}\prod_{j=1}^q\sin(\theta_j)&\text{if }k=1,\\\cos(\theta_{k-1})\prod_{j=k}^q\sin(\theta_j)&\text{if }2\leq k\leq q,\\\cos(\theta_q)&\text{if }k=q+1,\end{cases}
\end{equation*}
where $\theta_1\in[-\piup,\piup]$ and $\theta_2,\dots,\theta_q\in [0,\piup]$.
In polar coordinates the surface measure reads as  
\begin{equation*}
\dd\mu_q=\prod_{k=1}^q\sin(\theta_k)^{k-1}\ \dd\theta_k=\sin(\theta_q)^{q-1}\ \dd\theta_q\, \dd\mu_{q-1},
\end{equation*}
or equivalently 
\begin{equation*}
\dd\mu_q=w_q(t) \dd t \dd\mu_{q-1},
\end{equation*}
with Jacobi weight function $w_q(t)=(1-t^2)^{\frac{q}{2}-1}$ and $t=\cos(\theta_q)$.\\ 

According to the weight $w_q$ the spaces $L_{w_q,p}([-1,1])\defeq L_p([-1,1],w_q(t)\, \dd t)$ are defined in the usual manner.
Using the above decomposition of $\dd\mu_q$ it can be easily seen that for any $\phi\in L_{w_q,1}([-1,1])$ and any $y\in\sph^q$ 
\begin{equation*}
\int_{\mathbb{S^q}}\phi(x\cdot y) \dd\mu_q(x)=\omega_{q-1}\int_{-1}^1\phi(t) w_q(t) \dd t.
\end{equation*}

Let $n\geq 0$ be a fixed integer.
The restriction of a harmonic homogeneous polynomial of degree $n$ to $\sph^q$ is called a \emph{spherical harmonic} of degree $n$.
Spherical harmonics of degree at most $n$ for a vector space $\Pi^q_n$.
The vector space of spherical harmonics of degree equal to $n$ shall be denoted by $\mathcal{H}_n^q$.
The spaces $\mathcal{H}_n^q$ are mutually orthogonal with respect to the inner product on $L_2(\sph^q)$ and, moreover, we have the following decomposition $\Pi_n^q=\bigoplus_{\ell=0}^n\mathcal{H}_\ell^q$.
Clearly, the spaces $\mathcal{H}_{\ell}^q$ are finite-dimensional and the dimension of $\Pi_n^q$ is given by the sum of the dimensions of the spaces $\mathcal{H}_\ell^q$, $\ell=0,\dots,n$.
More precisely,
\begin{equation*}
\dim(\mathcal{H}_{\ell}^q)=\frac{(2\ell+q-1)(\ell+q-2)}{\ell!\ (q-1)!}\eqdef h_q(\ell),\qquad \dim(\Pi_n^q)=\sum_{\ell=0}^n h_q(\ell)\eqdef d_q(n).
\end{equation*}
Let $\setcond{Y_{n,k}}{k=1,\dots, h_q(n)}$ be an orthonormal basis for $\mathcal{H}_n^q$.
The following relation of the basis elements $Y_{n,k}$ to the ultraspherical polynomials, known as the addition formula, is of fundamental importance to our analysis  
\begin{equation}
\label{eq:additionFormular}
\sum_{k=1}^{h_q(n)} Y_{n,k}(x)\, \overline{Y_{n,k}(y)}=\frac{h_q(n)}{\omega_q}\ R_n^{(\frac{q}{2}-1,\frac{q}{2}-1)}(x\cdot y),
\end{equation}
where $R_n^{(\frac{q}{2}-1,\frac{q}{2}-1)}$ is the ultraspherical polynomial corresponding to the weight $w_q$ and normalized such that 
$R_n^{(\frac{q}{2}-1,\frac{q}{2}-1)}(1)=1$.
The orthogonality relation for these polynomials reads as 
\begin{equation*}
\int_{-1}^1 R_n^{(\frac{q}{2}-1,\frac{q}{2}-1)}(t)\, R_m^{(\frac{q}{2}-1,\frac{q}{2}-1)}(t)\ w_q(t)\, \dd t=\frac{\omega_q}{\omega_{q-1}\, h_q(n)}\ \delta_{n,m}.
\end{equation*}
In order to simplify the notation we will write $R_n$ instead of $R_n^{(\frac{q}{2}-1,\frac{q}{2}-1)}$.\\ 

The space $L_2(\sph^q)$ can be decomposed in terms of the spaces $\mathcal{H}_n^q$ as 
\begin{equation*}
L_2(\sph^q)=c\ell\bigoplus_{n\in\mathbb{N}_0}\mathcal{H}_n^q=c\ell\ \mathrm{span}\setcond{Y_{n,k}}{n\in\mathbb{N}_0,\ k=1,\dots, h_q(n)}.
\end{equation*}
Consequently, the orthogonal projection of $f\in L_2(\sph^q)$ onto $\mathcal{H}_n^q$ reads as 
\begin{equation*}
\mathcal{P}_nf(x)=\sum_{k=1}^{h_q(n)}\skpr{f}{Y_{n,k}} Y_{n,k}(x)=\frac{h_q(n)}{\omega_q}\int_{-1}^1 f(y)\, R_n(x\cdot y)\ \dd\mu_q(y),
\end{equation*}
where the second identity is an implication of the addition formula \eqref{eq:additionFormular}.
The orthogonal projection onto the space $\Pi_n^q$ is therefore 
given as 
\begin{equation*}
\mathcal{S}_nf(x)=\sum_{k=0}^n\mathcal{P}_kf(x)=\frac{1}{\omega_q}\int_{\sph^q} f(y)\, K_n(x\cdot y,1)\ \dd\mu_q(y),
\end{equation*}
where 
\begin{equation}
\label{eq:Darboux}
K_n(t,t^\prime)=\sum_{k=0}^n \mnorm{R_k}^{-2}_2 R_k(t) R_k(t^\prime)
\end{equation}
is the Christoffel--Darboux kernel for the ultraspherical polynomials and $\mnorm{R_k}_2=(\int_{-1}^1\abs{R_k(t)}^2\dd t)^{\frac{1}{2}}$.
In order to simplify our notation we will write $K_n(t)$ for
$K_n(t,1)$.

In the one-dimensional case, i.e., on $\sph^1 = \TT = \RR/\ZZ$, the $(2n+1)$-dimensional polynomial spaces $\Pi_n^1 = T(n)$ consists of the trigonometric polynomials $f:\TT \to \CC$, $f(x)=\sum_{k=-n}^n c_k\exp(\ii k x)$, where $c_{-n},\ldots,c_n\in\CC$.
The well-known Bernstein inequality for trigonometric polynomials reads as follows, see \cite[Theorem~III.3.16]{Zygmund2002}.
\begin{Lem}\label{thm:bernstein}
For $p\in [1,\infty]$ and $P\in T(n)$, we have
\begin{equation*}
\mnorm{P^\prime}_{\TT,p}\leq n\mnorm{P}_{\TT,p},
\end{equation*}
where the $L_p$ norm of a function defined on the torus is given as
\begin{equation}
\mnorm{f}_{\TT,p}\defeq\begin{cases}\lr{(2\piup)^{-1}\int_{-\piup}^\piup \abs{f(x)}^p\dd x}^{\frac{1}{p}}&\text{if }1\leq p<\infty,\\\sup_{x\in [-\piup,\piup]}\abs{f(x)}&\text{if }p=\infty.\end{cases}\label{eq:torus-norm}
\end{equation}
\end{Lem}

Analogously to \cref{eq:torus-norm}, we may define
\begin{equation*}
\mnorm{f}_{\mu,p}\defeq\begin{cases}\lr{\int_{\sph^q}\abs{f(x)}^p\dd\mu(x)}^{\frac{1}{p}}&\text{if }1\leq p <\infty,\\\mu\text{-}\esssup\setcond{\abs{f(x)}}{x\in \sph^q}&\text{if } p=\infty.\end{cases}
\end{equation*}
and
\begin{equation*}
\mnorm{x}_{\RR^D,p}\defeq\begin{cases}\lr{\sum_{j=1}^D \abs{x_j}^p}^{\frac{1}{p}}&\text{if }1\leq p <\infty,\\\sup\setn{\abs{x_1},\ldots,\abs{x_D}}&\text{if }p=\infty.\end{cases}
\end{equation*}
for any measure $\mu$ on $\sph^q$, $f:\sph^q\to\CC$, and $x=(x_1,\ldots,x_D)\in\RR^D$.

For our analysis it will be necessary to consider partitions of $\sph^q$ and related sets of points.
A family $\prt=\setn{\Z_1,\dots, \Z_N}$ of measurable subsets $\Z_k\subseteq\sph^q$ is called a partition of $\sph^q$ if their interiors are pairwise disjoint, i.e., $\inte(\Z_k)\cap\inte(\Z_{k^\prime})=\emptyset$ for all $k,k^\prime\in\setn{1,\dots,N}$, and $\sph^q=\bigcup_{k=1}^N \Z_k$.
An element $\Z\in\prt$ is called a \emph{patch}.
A finite subset $\smpl$ of $\sph^q$ is called compatible with the partition if there is precisely one element of $\smpl$ in the interior of every patch of $\prt$, viz.\ $\smpl\cap \Z_k=\setn{\xi}$ for every $k\in\setn{1,\ldots,N}$.
We will call the pair $(\smpl,\prt)$ compatible if the set $\smpl$ is compatible with the partition $\prt$ and we will write $\Z_\xi$ to indicate the
patch from $\prt$ which contains the element $\xi\in\smpl$.
There are two parameters related to $\smpl$ resp.\ $\prt$ which will be relevant for our analysis.
These are the \emph{mesh norm} of $\smpl$ defined as 
\begin{equation*}
\delta_\smpl\defeq\max_{x\in\sph^q}\,\min_{y\in\smpl}\, d(x,y),
\end{equation*}
and the \emph{partition norm} related to $\prt$ given by 
\begin{equation*}
\mnorm{\prt}\defeq\max_{\Z\in\prt}\,\max_{x,y\in \Z}\, d(x,y).
\end{equation*}
\begin{figure}[ht]
\begin{center}
\includegraphics[scale=0.3]{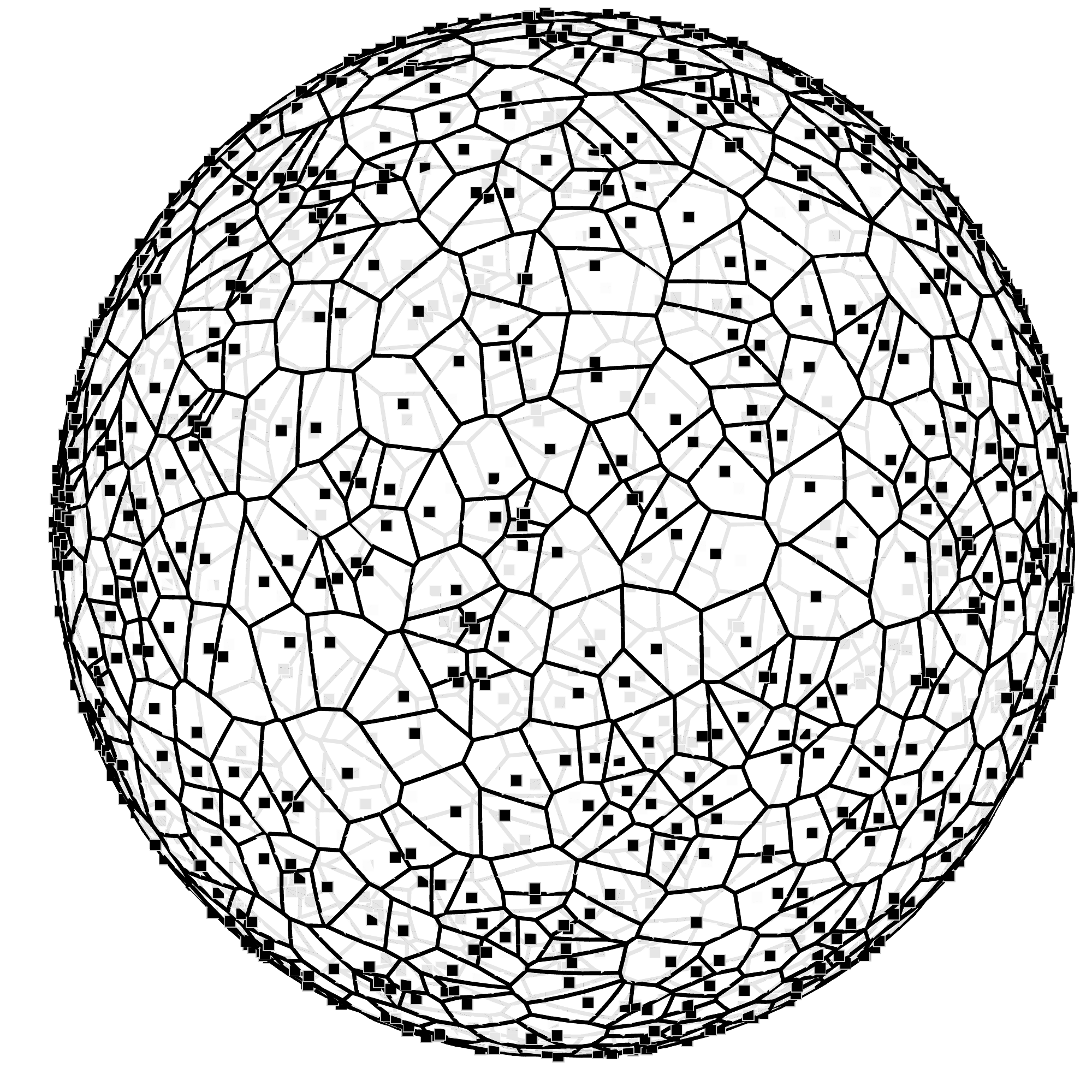}
\end{center}
\caption{Example of a compatible pair $(\smpl,\prt)$.}
\label{fig:partition}
\end{figure}

In view of the Marcinkiewicz--Zygmund inequality, we discretize a measure $\mu$ on $\sph^q$ using the data of a compatible pair $(\smpl,\prt)$ by $\mu(\smpl,\prt)\defeq\sum_{\xi\in\smpl}\mu(\Z_\xi)\delta_{\xi}$, where $\delta_\xi(A)\defeq\mathbf{1}_A(\xi)$ denotes the Dirac measure at $\xi\in\smpl$.

Before we get to $L_p$-Marcinkiewicz--Zygmund inequalities for general $1\leq p\leq \infty$ later, we have a look at the special case $p=2$ through the lens of random matrix theory in the next section.
These proof techniques are tailored to the $p=2$ case, and yield a first version of the Marcinkiewicz--Zygmund inequality for randomly chosen sampling points.


\section{A first look at random points}
\label{sec:random_first}
In this section, we consider the following randomized setting.
Let $\smpl=\setn{\xi_1,\ldots,\xi_N}$ be a set of points on $\sph^q$ 
drawn i.i.d.\ according to the normalized surface area measure $\sigma_q=\frac{1}{\omega_q}\mu_q$ on $\sph^q$.
The aim is to provided a relationship between the number $N$ of samples, the dimension $q$, the polynomial degree $n$, and the parameter $\eta\in \mathopen]0,1\mathclose[$ such that
\begin{equation}
(1-\eta)\mnorm{P}_{\sigma_q,2}^2\leq \frac{1}{N}\sum_{j=1}^N \abs{P(\xi_j)}^2\leq (1+\eta)\mnorm{P}_{\sigma_q,2}^2,\label{eq:mz-squared-normalized}
\end{equation}
holds with high probability for every $P\in\Pi_n^q$.

To keep the notation simple we will write $d$ instead of $d_q(n)$ for the dimension of $\Pi_n^q$.
Let $(e_k)_{k=1}^d$ be an orthonormal basis
of $\Pi_n^q$ with respect to the inner product 
\begin{equation*}
\skpr{f}{g}\defeq\int_{\sph^q}f(x)\overline{g(x)}\dd\sigma_q(x).
\end{equation*}
Parseval's identity yields
\begin{equation*}
\mnorm{P}_{\sigma_q,2}=\mnorm{x}_{\RR^{d},2}
\end{equation*}
where $P\in\Pi_n^q$ and $x=(\skpr{f}{e_k})_{k=1}^d$.
Now consider 
\begin{equation}
L=\begin{pmatrix}
e_1(\xi_1)&e_2(\xi_1)&\cdots&e_M(\xi_1)\\
\vdots&\vdots&\ddots&\vdots\\
e_1(\xi_N)&e_2(\xi_N)&\cdots&e_M(\xi_N)\\
\end{pmatrix},\label{eq:matrix}
\end{equation}
and note that 
\begin{equation*}
(Lx)_j=\sum_{k=1}^d\skpr{f}{e_k}e_k(\xi_j)=f(\xi_j)
\end{equation*}
for $j=1,\ldots,N$.
Thus the inequality \eqref{eq:mz-squared-normalized} can be rewritten as 
\begin{equation}\label{eq:singular-values}
(1-\eta)\mnorm{x}_{\RR^d,2}^2\leq \mnorm{\frac{1}{\sqrt{N}}Lx}_{\RR^N,2}^2\leq (1+\eta)\mnorm{x}_{\RR^d,2}^2.
\end{equation}
Obviously, the best possible constants $1\pm \eta$ in \eqref{eq:singular-values} are given by the minimal resp.\ maximal eigenvalue of
$\frac{1}{N}L^\ast L$.
In \cite[Theorem~2.1]{MoellerUl2021}, Moeller and Ullrich proved the following concentration inequality for the smallest and largest eigenvalue of such random Gram matrices.
The result is based on Tropp \cite{Tropp2012}.
\begin{Satz}\label{thm:moeller-ullrich}
Let $s,N,M\in\NN$, $t\in \mathopen]0,1\mathclose[$, $\Omega\subseteq\RR^s$ a set, $\varrho$ a probability measure on $\Omega$ and $(e_k)_{k=1}^D$ be an orthonormal system in $L_2(\Omega,\varrho)$.
Let $\xi_1,\ldots,\xi_N\in \Omega$ be drawn i.i.d.\ according to $\varrho$, $L=(e_k(\xi_j))_{j,k=1}^{N,D}$, and $\PP=\varrho^{\otimes N}$ the product measure.
Then the following concentration inequalities for the extremal eigenvalues of $\frac{1}{N}L^\ast L$ hold 
\begin{align*}
\mathbb{P}\lr{\lambda_{\min}\lr{\frac{1}{N}L^\ast L}<1-t}&<(D+1)\exp\lr{-\frac{N\log((1-t)^{1-t}\ee^t)}{\sup_{x\in \Omega}\sum_{k=1}^D\abs{e_
k(x)}^2}},\\
\mathbb{P}\lr{\lambda_{\max}\lr{\frac{1}{N}L^\ast L}>1+t}&<(D+1)\exp\lr{-\frac{N\log((1+t)^{1+t}\ee^{-t})}{\sup_{x\in \Omega}\sum_{k=1}^D\abs{e_k(x)}^2}}.
\end{align*}
\end{Satz}

To apply this result to our case let $s=q+1$, $\Omega=\sph^q$, $\varrho=\sigma_q$, and let $(e_k)_{k=1}^d$ be an orthonormal basis of the Hilbert space $(\Pi_n^q,\mnorm{\cdot}_{\sigma_q,2})$.

To compute the expression $\sup_{x\in \sph^q}\sum_{k=1}^d\abs{e_k(x)}^2$, note that orthonormal bases of $(\Pi_n^q,\mnorm{\cdot}_{\sigma_q,2})$
are obtained from orthonormal bases of $(\Pi_n^q,\mnorm{\cdot}_{\mu_q,2})$ by multiplying each element by the constant scalar
$\omega_q^{\frac{1}{2}}$.
Using the addition formula \eqref{eq:additionFormular} an easy computation shows that 

\begin{equation}
\sum_{k=1}^d\abs{e_k(x)}^2=d = \sum_{\ell=0}^n \frac{(2\ell+q-1)(\ell+q-2)!}{\ell!(q-1)!}\label{eq:christoffel}
\end{equation}
for every $x\in\sph^q$.

\begin{Satz}\label{thm:tropp-mz}
Let $\eta,\eps\in \mathopen]0,1\mathclose[$.
Suppose $\xi_1,\ldots,\xi_N\in \sph^q$ are drawn i.i.d.\ according to $\sigma_q$.
If 
\begin{equation*}
N>\log\lr{\frac{2d_q(n)}{\eps}}\frac{3d_q(n)}{\eta^2},
\end{equation*}
then with probability exceeding $1-\eps$ with respect to the product measure $\PP=\sigma_q^{\otimes N}$, we have
\begin{equation*}
(1-\eta)\mnorm{P}_{\sigma_q,2}^2\leq \frac{1}{N}\sum_{j=1}^N \abs{P(\xi_j)}^2\leq (1+\eta)\mnorm{P}_{\sigma_q,2}^2
\end{equation*}
for all $P\in\Pi_n^q$.
\end{Satz}

\begin{proof}
We will again use $d$ for $d_q(n)$.
Let $(e_k)_{k=1}^d$ be an orthonormal basis for $(\Pi_n^q,\mnorm{\cdot}_{\sigma_q})$ and $L\defeq (e_k(\xi_j))_{j,k=1}^{N,d}$.
By \cref{thm:moeller-ullrich} and \eqref{eq:christoffel}, we have
\begin{equation}
\mathbb{P}\lr{\lambda_{\min}\lr{\frac{1}{N}L^\ast L}<1-\eta}<d\exp\lr{-\frac{N\log((1-\eta)^{1-\eta}\ee^t)}{d}}\label{eq:lambda-min}
\end{equation}
and
\begin{equation}
\mathbb{P}\lr{\lambda_{\max}\lr{\frac{1}{N}L^\ast L}>1+\eta}<d\exp\lr{-\frac{N\log((1+\eta)^{1+\eta}\ee^{-\eta})}{d}}.\label{eq:lambda-max}
\end{equation}
For $\eta\in \mathopen]0,1\mathclose[$, we have
\begin{equation*}
\max\setn{\log((1-\eta)^{1-\eta}\ee^\eta), \log((1+\eta)^{1+\eta}\ee^{-\eta})}\geq\frac{\eta^2}{3},
\end{equation*}
so the right-hand sides of \eqref{eq:lambda-min} and \eqref{eq:lambda-max} are each less or equal to $\exp\lr{-\frac{N\eta^2}{3d}}$.
Note that $N>\log\lr{\frac{2d}{\eps}}\frac{3d}{\eta^2}$ is equivalent to $d\exp\lr{-\frac{N\eta^2}{3d}}<\frac{\eps}{2}$.
Thus
\begin{align*}
&\norel\mathbb{P}\lr{\lambda_{\min}\lr{\frac{1}{N}L^\ast L}<1-\eta}+\mathbb{P}\lr{\lambda_{\max}\lr{\frac{1}{N}L^\ast L}>1+\eta}\\
&<2d\exp\lr{-\frac{N\eta^2}{3d}}\\
&<\eps.
\end{align*}
This concludes the proof.
\end{proof}
Note that the statement \cref{thm:tropp-mz} holds verbatim for any direct sum $\bigoplus_{\ell\in J}\mathcal{H}_\ell^q$ for some index set $J\subseteq \NN$ in place of $\Pi_n^q$, with $d_{n,q}$ replaced by $\dim(\bigoplus_{\ell\in J}\mathcal{H}_\ell^q)$.
Like in \eqref{eq:christoffel}, this is due to fact that the addition formula for orthonormal bases holds true for the summands $H_l^q$, see again \cite[equation~(2.8)]{FilbirTh2008}.

In order to illustrate \cref{thm:tropp-mz} we fix $q=2$, $\eta=0.9$ and $\varepsilon = 0.01$, randomly draw $N=\log\lr{\frac{2d_q(n)}{\eps}}\frac{3d_q(n)}{\eta^2}$ spherical points $\xi_1,\ldots,\xi_N \in \mathbb S^2$ and compute
the minimum and maximum eigenvalues $\lambda_{\text{min}}$ and $\lambda_{\text{max}}$ of the matrix $\frac{1}{N} L^\ast L$.
This we repeated 1000 times for the different polynomials degrees $n$ and depicted in \cref{fig:MZProp2} the average minimum and maximum eigenvalues as well as the 1 percent and 99 quantiles.
According to our experiment those are safely within the range $[1-\eta,1+\eta]$ as stated by \cref{thm:tropp-mz}.

\pgfplotstableread{
    N    lMin1     lMinM     lMin99    lMax1     lMaxM     lMax99
    8    0.5556    0.6566    0.7240    1.3021    1.3923    1.5388
   10    0.5559    0.6502    0.7170    1.3213    1.4043    1.5530
   12    0.5542    0.6459    0.7045    1.3242    1.4112    1.5530
   16    0.5386    0.6372    0.6938    1.3373    1.4217    1.5650
   20    0.5599    0.6356    0.6877    1.3515    1.4288    1.5498
   25    0.5381    0.6304    0.6783    1.3574    1.4348    1.5495
   32    0.5481    0.6269    0.6736    1.3719    1.4428    1.5567
   40    0.5520    0.6228    0.6697    1.3820    1.4472    1.5598
   50    0.5498    0.6196    0.6623    1.3935    1.4554    1.5640
}{\eigMZ}
    
\begin{figure}
    \centering
    \begin{tikzpicture}
    \begin{semilogxaxis}[%
        xtick=data, log ticks with fixed point,
        xlabel={polynomial degree $n$},
        ylabel={$\lambda_{\text{min}}$,$\lambda_{\text{max}}$}]

        \addplot+[forget plot,only marks, mark size=3pt, draw=black,mark options={scale=1, draw=black,fill=black}] 
        plot[error bars/.cd, y dir=both, y explicit,  error mark options={  rotate=90,black,mark size=3pt,line width=1pt}]
        table[x=N,y=lMinM,y error minus expr=\thisrow{lMinM}-\thisrow{lMin1},
        y error plus expr=\thisrow{lMin99}-\thisrow{lMinM}] {\eigMZ};

        \addplot+[forget plot,only marks, mark size=3pt, draw=black,mark options={scale=1, draw=black,fill=black}] 
        plot[error bars/.cd, y dir=both, y explicit,  error mark options={  rotate=90,black,mark size=3pt,line width=1pt}]
        table[x=N,y=lMaxM,y error minus expr=\thisrow{lMaxM}-\thisrow{lMax1},
        y error plus expr=\thisrow{lMax99}-\thisrow{lMaxM}] {\eigMZ};
    \end{semilogxaxis}
    \end{tikzpicture}
    \caption{Concentration of the minimum and maximum eigenvalues of the matrix $\frac{1}{N} L^\ast L$ for random sample sets $\xi_1,\ldots,\xi_N \in \mathbb S^2$ and different polynomial degrees $n$. The number of sampling points $N$ is chosen according to the lower bound in \cref{thm:tropp-mz}, where we have used the constants $\eta=0.9$ and $\eps=0.01$. Displayed are the mean minimum and maximum eigenvalues as well as the 1 and 99 percent quantiles.}
    \label{fig:MZProp2}
\end{figure}
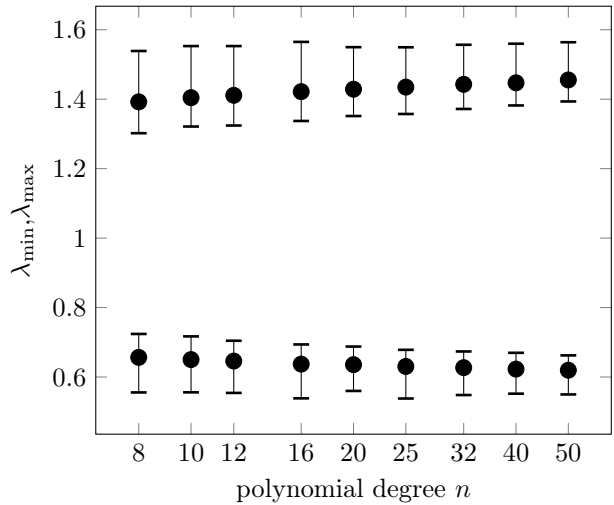

As a first step towards general $p$, we obtain the following statement when $p$ is even.
\begin{Kor}
Let $\eta,\eps\in \mathopen]0,1\mathclose[$ and let $p\in\NN$ be an even number.
Assume $\xi_1,\ldots,\xi_N\in \sph^q$ are drawn i.i.d.\ according to $\sigma_q$.
If 
\begin{equation*}
N>\log\lr{\frac{2d_q(np/2)}{\eps}}\frac{3d_q(np/2)}{\eta^2},
\end{equation*}
then with probability exceeding $1-\eps$ with respect to the product measure $\PP=\sigma_q^{\otimes N}$, we have
\begin{equation*}
(1-\eta)\mnorm{P}_{\sigma_q,p}^p\leq \frac{1}{N}\sum_{j=1}^N \abs{P(\xi_j)}^p\leq (1+\eta)\mnorm{P}_{\sigma_q,p}^p,
\end{equation*}
for all $P\in\Pi_n^q$.
\end{Kor}

\begin{proof}
For abbreviation we put $\tilde{d} = d_q(np/2)$.
Let $(e_k)_{k=1}^{\tilde{d}}$ be an orthonormal basis of the Hilbert space $(\Pi_{\frac{np}{2}}^q,\mnorm{\cdot}_{\sigma_q,2})$.
Since with $P \in \Pi^q_n$ it holds $P^\frac{p}{2} \in \Pi^q_{\frac{n p}{2}}$ we have by  \cref{thm:tropp-mz} for $N>\log\lr{\frac{2\tilde{d}}{\eps}}\frac{3\tilde{d}}{\eta^2}$
\begin{equation*}
(1-\eta)\mnorm{P^{\frac{p}{2}}}_{\sigma_q,2}^2\leq \frac{1}{N}\sum_{j=1}^N \abs{P(\xi_j)^{\frac{p}{2}}}^2
\leq (1+\eta)\mnorm{P^{\frac{p}{2}}}_{\sigma_q,2}^2
\end{equation*}
with probability $\geq 1-\eps$ with respect to 
the product measure $\PP=\sigma_q^{\otimes N}$.

This is equivalent to the assertion.
\end{proof}

In order to obtain Marcinkiewicz--Zygmund inequalities for random sampling points and general $p\in [1,\infty]$ we first reconsider the case were 
the sampling points are deterministic scattered points on $\sph^q$.


\section{Marcinkiewicz--Zygmund inequalities for scattered data}\label{sect:scattered}
In this section, we give a proof for a deterministic Marcinkiewicz--Zygmund inequality on $\sph^q$ which holds for all $p$ simultaneously.
Reasoning from the Riesz--Thorin interpolation theorem has been attempted in the literature several times, however (to our best knowledge) always fraught with problems.
The authors of \cite{FilbirMh2011} are aware of this issue and prove deterministic Marcinkiewicz--Zygmund inequality in a manifold setting by different means.
The aim of this section is to provide a self-contained and rather elementary proof for the sphere by proper use of Riesz--Thorin interpolation, which, in addition, simplifies some of the technical calculations in \cite{FilbirTh2008,MhaskarNaWa2001}.
The main theorem in this section reads as follows.
\begin{Satz}\label{thm:deterministic-mz}
Let $\eta\in \mathopen]0,1\mathclose[$, and let $(\smpl,\prt)$ be a compatible pair consisting of a finite set $\smpl\subseteq \sph^q$ and a partition $\prt$ of $\sph^q$.
Assume that
\begin{equation*}
5C_q(n+q^2)\mnorm{\prt}\leq \eta
\end{equation*}
with $C_q\defeq 2(3+3^{q/2}\,\piup)$.\label{eq:constant}
Then, for all $p\in[1,\infty]$ and every $P\in\Pi_n^q$, we have
\begin{equation*}
(1-\eta)\mnorm{P}_{\mu_q, p}\leq\mnorm{P}_{\mu_q(\smpl,\prt),p}\leq (1+\eta)\mnorm{P}_{\mu_q,p}.
\end{equation*}
\end{Satz}
The proof of \cref{thm:deterministic-mz} is essentially based on a generalized de la Vall\'ee Poussin kernel $v_n \colon [-1,1] \to \RR$, $n \in \NN$ for the system of ultraspherical
polynomials, which was defined in \cite{FilbirTh2008} as
\begin{equation*}
v_n(t)=\frac{1}{w_{q-1}}\frac{K_{\floor{\frac{n}{2}}}(t)K_{\floor{\frac{3n}{2}}}(t)}{K_{\floor{\frac{n}{2}}}(1)},
\end{equation*}
where $K_n$ is the Christoffel--Darboux kernel defined in \eqref{eq:Darboux}.
The generalized de la Vallée Poussin kernel $v_n$ is a polynomial of degree $2n$ that reproduces polynomials $P\in\Pi_n^q$, up to degree $n$ viz.\
\begin{equation*}
    P(x)=\int_{\sph^q} P(y)\, v_n(x\cdot y)\ \dd\mu_q(y).
\end{equation*} 
as it does the Christoffel--Darboux kernel $K_n$.
Additionally, the kernels $v_n$, $n \in \NN$ have bounded $L_1$ norm
 \begin{equation}
\int_{-1}^1\abs{v_n(t)}(1-t^2)^{\frac{q}{2}-1}\dd t\leq\frac{3^{\frac{q}{2}}}{\omega_{q-1}}\label{eq:ell-1-dlvp}
\end{equation}
and satsify
\begin{equation}
\label{eq:ell-infty-dlvp}
\sup_{t\in[-1,1]}\abs{v_n(t)}\leq\frac{1}{\omega_{q-1}}\frac{2^{-q+1}(\floor{\frac{3n}{2}}+q)^q}{\Gamma(\frac{q}{2})\Gamma(\frac{q}{2}+1)}
\leq\frac{1}{\omega_{q-1}}\frac{2\max\setn{n,2q}^q}{\Gamma(\frac{q}{2})\Gamma(\frac{q}{2}+1)}.
\end{equation}
For the proof of these statements we refer to \cite[Section 3.3]{FilbirTh2008}.
We prepare the proof of \cref{thm:deterministic-mz} by first showing  an integral bound of the derivative of the generalized de la Vall\'ee Poussin kernel.
\begin{Lem}\label{thm:dvlp}
For $n,q\in\NN$ with $q\geq 2$, the following estimate holds
\begin{equation*}
\int_0^\piup\abs{v_n^\prime(\cos(\tau))\sin(\tau)^q}\dd\tau\leq C_q\frac{(n+q^2)}{\omega_{q-1}}.
\end{equation*}
where $C_q=3^\frac{q}{2}\piup+2q+2$.
\end{Lem}
\begin{proof}
Let $\theta\defeq 1/\max\setn{n,2q}$.
We split the integral over $[0,\piup]$ into three parts
\begin{equation*}
\int_0^\piup\abs{v_n^\prime(\cos(\tau))\sin(\tau)^q}\dd\tau=\Big(\int_0^{\theta}+\int_{\theta}^{\piup-\theta}+\int_{\piup-\theta}^\piup\Big) \abs{v_n^\prime(\cos(\tau))\sin(\tau)^q}\dd\tau.
\end{equation*}
By the trigonometric Bernstein inequality (\cref{thm:bernstein}) and \cref{eq:ell-infty-dlvp}, we obtain for the first integral
\begin{align*}
A_1&\defeq \int_0^{\theta}\abs{v_n^\prime(\cos(\tau))\sin(\tau)^q}\dd\tau
=\int_0^{\theta} \abs{(v_n\circ \cos)^\prime(\tau)\sin(\tau)^{q-1}}\dd \tau\\
&\leq 2n \mnorm{v_n\circ \cos}_{\TT,\infty} \int_0^{\theta}\tau^{q-1}\dd \tau
\le \frac{4n}{\omega_{q-1} q\, \Gamma(\frac{q}{2})\Gamma(\frac{q}{2}+1)}  
\le \frac{2n}{\omega_{q-1}}
\end{align*}
Thanks to symmetry the same upper bounds is valid for the third integral
\begin{align*}
A_3 \defeq \int_{\piup-\theta}^{\piup}
\abs{v_n^\prime(\cos(\tau))\sin(\tau)^q}\dd\tau.
\end{align*}

Using the product rule followed by triangular inequality we split the middle integral into 
\begin{align*}
A_2&\defeq \int_{\theta}^{\piup-\theta}\abs{v_n^\prime(\cos(\tau))\sin(\tau)^q}\dd\tau\\
&\leq \underbrace{\int_{\theta}^{\piup-\theta} \hspace{-0.2cm}\abs{((v_n\circ\cos)\sin^{q-1})^\prime(\tau)} \dd\tau}_{\eqdef I_1}
 +\underbrace{\int_{\theta}^{\piup-\theta}\hspace{-0.2cm}\abs{v_n(\cos(\tau))(q-1)\sin(\tau)^{q-2}\cos(\tau)}\dd\tau}_{\eqdef I_2}
\end{align*}
Applying the trigonometric Bernstein inequality to $(v_n \circ \cos) \sin^{q-1}$ we obtain in conjunction with \cref{eq:ell-1-dlvp}
\begin{align*}
I_1&\leq (2n+q-1)\int_0^\piup\abs{v_n(\cos(\tau))\sin(\tau)^{q-1}}\dd \tau
\leq (2n+q-1)\frac{3^{\frac{q}{2}}}{\omega_{q-1}}\\
\intertext{and }
I_2&=(q-1)\int_{\theta}^{\piup-\theta}\abs{v_n(\cos(\tau))}\sin(\tau)^{q-1}\frac{\abs{\cos(\tau)}}{\sin(\tau)}\dd\tau\\
&\leq (q-1)\frac{\piup}{2\,\theta}\int_{\theta}^{\piup-\theta}\abs{v_n(\cos(\tau))\sin(\tau)^{q-1}}\dd\tau
\leq (q-1)\frac{\piup}{2\,\theta}\frac{3^{\frac{q}{2}}}{\omega_{q-1}},
\end{align*}
where we made use of $\abs{\cos(\tau)}\leq 1$ for all $\tau\in\RR$, $\frac{1}{\sin(\tau)}\leq \frac{\piup}{2\tau}\leq\frac{\piup}{2\,\theta}$ when $\theta\leq \tau\leq\frac{\piup}{2}$, and $\frac{1}{\sin(\tau)}\leq\frac{\piup}{2(\piup-\tau)}\leq\frac{\piup}{2\,\theta}$ when $\frac{\piup}{2}\leq\tau\leq\piup-\theta$.
Finally we arrive at 
\begin{align*}
\int_0^\piup \abs{v_n^\prime(\cos(\tau))\sin(\tau)^q}\dd\tau
&\leq A_1+A_2+I_1+I_2\\
&\leq \frac{4n}{\omega_{q-1}}+ \frac{3^{q/2}}{\omega_{q-1}} (2n + (q-1)(1+\piup \max\lr{\frac{n}{2},q)} \\
&\le \frac{4 + 2 \cdot 3^{\frac{q}{2}} + \frac{q-1}{2}\piup}{\omega_{q-1}} n
+ \frac{3^{q/2} (q-1)(1+\piup q)}{\omega_{q-1}}\\
&\leq\frac{3^\frac{q}{2}\piup+2q+2}{\omega_{q-1}}(n+q^2)
\end{align*}
which concludes the proof.
\end{proof}

A key step in the proof of \cref{thm:deterministic-mz} is to show that for every compatible pair $(\smpl,\prt)$ 
\begin{equation}
T_{\smpl,\prt,n}(f)(x)\defeq\int_{\sph^q}\sum_{\xi\in\smpl}\mathbf{1}_{\Z_\xi}(x)\big(v_n(x\cdot y)-v_n(\xi\cdot y)\big)f(y)\dd\mu_q(y),\label{eq:linear-operator}
\end{equation}
defines a bounded operator $T_{\smpl,\prt,n}:L_p(\sph^q,\mu_q)\to L_p(\sph^q,\mu_q)$ for all $p\in [1,\infty]$.
We concentrate on the extreme cases $p=1$ and $p=\infty$ in the following two lemmas and start with $p=1$.
\begin{Lem}\label{thm:bounded-operator-ell-1}
The mapping $T_{\smpl,\prt,n}$ defines a bounded linear operator from $L_1(\sph^q)$ to $L_1(\sph^q)$ with norm
\begin{equation*}
\mnorm{T_{\smpl,\prt,n}}_{1\to 1}\leq \lr{\frac{2^{q+3}}{q\Gamma(\frac{q}{2})\Gamma(\frac{q}{2}+1)}+4 C_q}(n+q^2)\mnorm{\prt}.
\end{equation*}
provided that $(n+q^2)\mnorm{\prt}\in \mathopen]0,1\mathclose[$.
\end{Lem}

\begin{proof}
Using the triangle inequality, Fubini's theorem, and Hölder's inequality, we obtain
\begin{align*}
&\norel\mnorm{T_{\smpl,\prt,n}(f)}_{\mu_q,1}\\
&=\int_{\sph^q}\abs{\int_{\sph^q}\sum_{\xi\in\smpl}\mathbf{1}_{\Z_\xi}(x)\big(v_n(x\cdot y)-v_n(\xi\cdot y)\big)\,f(y)\dd\mu_q(y)}\dd\mu_q(x)\\
&\leq\int_{\sph^q}\int_{\sph^q}\sum_{\xi\in\smpl}\mathbf{1}_{\Z_\xi}(x)\abs{v_n(x\cdot y)-v_n(\xi\cdot y)}\,\abs{f(y)}\dd\mu_q(y)\dd\mu_q(x)\\
&\leq \mnorm{f}_{\mu_q,1}\esssup_{y\in\sph^q}\int_{\sph^q}\sum_{\xi\in\smpl}\mathbf{1}_{\Z_\xi}(x)\abs{v_n(x\cdot y)-v_n(\xi\cdot y)}\dd\mu_q(x).
\end{align*}
Now fix $y\in\sph^q$.
The fundamental theorem of calculus and the triangle inequality give
\begin{align*}
&\norel\int_{\sph^q}\sum_{\xi\in\smpl}\mathbf{1}_{\Z_\xi}(x)\abs{v_n(x\cdot y)-v_n(\xi\cdot y)}\dd\mu_q(x)\\
&=\int_{\sph^2}\sum_{\xi\in\smpl}\mathbf{1}_{\Z_\xi}(x)\abs{\int_{[d(x,y),d(\xi,y)]}(v_n\circ\cos)^\prime (t)\dd t}\dd\mu_q(x)\\
&\leq\int_{\sph^q}\sum_{\xi\in\smpl}\mathbf{1}_{\Z_\xi}(x)\int_{d(x,y)-\mnorm{\prt}}^{d(x,y)+\mnorm{\prt}}\abs{(v_n\circ\cos)^\prime (t)}
  \dd t\, \dd\mu_q(x).\\
\intertext{Now integration is independent of $\xi$, and since $\sum_{\xi\in\smpl}\mathbf{1}_{\Z_\xi}(x)=1$ for $\mu_q$-almost all $x\in\sph^q$, this factor can be omitted.
Parametrizing $\sph^q$ with north pole $y$ yields}
&=\omega_{q-1}\int_0^\piup\sin(\tau)^{q-1}\int_{\tau-\mnorm{\prt}}^{\tau+\mnorm{\prt}}\abs{(v_n\circ\cos)^\prime(t)}\dd t\dd\tau
\end{align*}
where we only resolved the outer integral in the last step.
Having in mind that $\mnorm{\prt}<\frac{\piup}{2}$, we split the integration over $[0, \piup]$ into pieces:
\begin{equation*}
\int_0^\piup\sin(\tau)^{q-1}\int_{\tau-\mnorm{\prt}}^{\tau+\mnorm{\prt}}\abs{ (v_n\circ\cos)^\prime(t)}\dd t\dd\tau=B_1+B_2+B_3.
\end{equation*}
Upper bounds for the summands
\begin{align*}
B_1&\defeq\int_0^{2\mnorm{\prt}}\sin(\tau)^{q-1}\int_{\tau-\mnorm{\prt}}^{\tau+\mnorm{\prt}}\abs{(v_n\circ\cos)^\prime(t)}\dd t\dd\tau\\
&\leq 2\mnorm{\prt}\mnorm{(v_n\circ\cos)^\prime}_{\TT,\infty}\int_0^{2\mnorm{\prt}}\sin(\tau)^{q-1}\dd\tau\\
&\leq 4n\mnorm{\prt}\mnorm{v_n\circ\cos}_{\TT,\infty}\int_0^{2\mnorm{\prt}}\tau^{q-1}\dd\tau\\
&= 4n\mnorm{\prt}\mnorm{v_n\circ\cos}_{\TT,\infty}q^{-1}2^q\mnorm{\prt}^q\\
&\leq n\max\setn{n,2q}^q\mnorm{\prt}^{q+1}\frac{2^{q+2}}{q\omega_{q-1}\Gamma(\frac{q}{2})\Gamma(\frac{q}{2}+1)}\\
&\leq (n+q^2)^{q+1}\mnorm{\prt}^{q+1}\frac{2^{q+2}}{q\omega_{q-1}\Gamma(\frac{q}{2})\Gamma(\frac{q}{2}+1)}\\
\intertext{and likewise}
B_2&\defeq\int_{\piup-2\mnorm{\prt}}^\piup\sin(\tau)^{q-1}\int_{\tau-\mnorm{\prt}}^{\tau+\mnorm{\prt}}\abs{\frac{\dd}{\dd t} (v_n\circ\cos)^\prime(t)}\dd t\dd\tau\\
&\leq 2\mnorm{\prt}\mnorm{(v_n\circ\cos)^\prime}_{\TT,\infty}\int_{\piup-2\mnorm{\prt}}^\piup\sin(\tau)^{q-1}\dd\tau\\
&\leq (n+q^2)^{q+1}\mnorm{\prt}^{q+1}\frac{2^{q+2}}{q\omega_{q-1}\Gamma(\frac{q}{2})\Gamma(\frac{q}{2}+1)}\\
\end{align*}
are due to \cref{thm:bernstein}, \cref{eq:ell-infty-dlvp} and $\sin(\tau)\leq \tau$ for all $\tau\in\RR$.

Now, if $\tau-\mnorm{\prt}\leq t\leq \tau+\mnorm{\prt}$ and $2\mnorm{\prt}\leq\tau\leq \piup-2\mnorm{\prt}$, we have $\mnorm{\prt}\leq t\leq\piup-\mnorm{\prt}$, and thus
\begin{align*}
\sin(\tau)&=\sin(\tau-t+t)=\sin(t)\cos(\tau-t)+\sin(\tau-t)\cos(t)\\
&\leq \sin(t)+\sin(\mnorm{\prt})\leq 2\sin(t).
\end{align*}
This yields the following upper bound for the third summand
\begin{align*}
B_3&\defeq\int_{2\mnorm{\prt}}^{\piup-2\mnorm{\prt}}\sin(\tau)^{q-1}\int_{\tau-\mnorm{\prt}}^{\tau+\mnorm{\prt}}\abs{(v_n\circ\cos)^\prime(t)}\dd t\dd\tau\\
&\leq 2\int_{2\mnorm{\prt}}^{\piup-2\mnorm{\prt}}\sin(t)^{q-1}\int_{\tau-\mnorm{\prt}}^{\tau+\mnorm{\prt}}\abs{(v_n\circ\cos)^\prime(t)}\dd t\dd\tau\\
&=2\int_{2\mnorm{\prt}}^{\piup-2\mnorm{\prt}}\int_{-\mnorm{\prt}}^{\mnorm{\prt}}\abs{v_n^\prime(\cos(t+\tau))}\sin(t+\tau)^q\dd t\dd\tau\\
&=2\int_{-\mnorm{\prt}}^{\mnorm{\prt}}\int_{2\mnorm{\prt}}^{\piup-2\mnorm{\prt}}\abs{v_n^\prime(\cos(t+\tau))}\sin(t+\tau)^q\dd\tau\dd t.\\
\intertext{Another change of variables and enlarging the integration interval yields}
&=2\int_{-\mnorm{\prt}}^{\mnorm{\prt}}\int_{2\mnorm{\prt}+t}^{\piup-2\mnorm{\prt}+t}\abs{v_n^\prime(\cos(\tau))}\sin(\tau)^q\dd\tau\dd t\\
&=2\int_{-\mnorm{\prt}}^{\mnorm{\prt}}\int_0^\piup\abs{v_n^\prime(\cos(\tau))}\sin(\tau)^q\dd\tau\dd t.\\
&=4\mnorm{\prt}\int_0^\piup\abs{ v_n^\prime(\cos(\tau))}\sin(\tau)^q\dd\tau\\
&\leq 4\mnorm{\prt} C_q\frac{n+q^2}{\omega_{q-1}},
\end{align*}
where we used \cref{thm:dvlp} for the last step.
Using $(n+q^2)\mnorm{\prt}<1$, we obtain
\begin{align*}
&\norel\mnorm{T_{\smpl,\prt,n}(f)}_{\mu_q,1}\\
&\leq \mnorm{f}_{\mu_q,1}\omega_{q-1}\Biggl(2(n+q^2)^{q+1}\mnorm{\prt}^{q+1}\frac{2^{q+2}}{q\omega_{q-1}\Gamma(\frac{q}{2})\Gamma(\frac{q}{2}+1)}\\
&\qquad +4\mnorm{\prt} C_q\frac{n+q^2}{\omega_{q-1}}\Biggr)\\
&= \mnorm{f}_{\mu_q,1}(n+q^2)\mnorm{\prt}\lr{\frac{2^{q+3}}{q\Gamma(\frac{q}{2})\Gamma(\frac{q}{2}+1)}+4 C_q}
\end{align*}
and the proof is finished.
\end{proof}
We now turn to the other boundary case $p=\infty$.
\begin{Lem}\label{thm:bounded-operator-ell-infty}
The mapping  $T_{\smpl,\prt,n}$ is a bounded linear operator from $L_\infty(\sph^q)$ to $L_\infty(\sph^q)$ with norm
\begin{equation*}
\mnorm{T_{\smpl,\prt,n}}_{\infty\to\infty}\leq 4C_q(n+q^2)\mnorm{\prt}.
\end{equation*}
\end{Lem}

\begin{proof}
The triangle inequality yields
\begin{align*}
&\norel\mnorm{T_{\smpl,\prt,n}(f)}_{\mu_q,\infty}\\
&=\mu_q\text{-}\esssup_{x\in\sph^q}\Big|\int_{\sph^q}\sum_{\xi\in\smpl}\mathbf{1}_{\Z_\xi}(x)(v_n(x\cdot y)-v_n(\xi\cdot y))f(y)\dd\mu_q(y)\Big|\\
&\leq \mu_q\text{-}\esssup_{x\in\sph^q}\Big|\int_{\sph^q}\sum_{\xi\in\smpl}\mathbf{1}_{\Z_\xi}(x)(v_n(x\cdot y)-v_n(\xi\cdot y))\dd\mu_q(y)\Big|\ \mnorm{f}_{\mu_q,\infty}\\
&\leq \mu_q\text{-}\esssup_{x\in\sph^q}\int_{\sph^q}\sum_{\xi\in\smpl}\mathbf{1}_{\Z_\xi}(x)\abs{(v_n(x\cdot y)-v_n(\xi\cdot
y))}\dd\mu_q(y)\ \mnorm{f}_{\mu_q,\infty}.
\end{align*}
For $\mu_q$-almost all $x\in\sph^q$, there exists a unique element $\xi\in \smpl$ with $x\in \Z_\xi$.
For such pairs $(x,\xi)$, the integral $\int_{\sph^q}\sum_{\xi\in\smpl}\mathbf{1}_{\Z_\xi}(x)\abs{(v_n(x\cdot y)-v_n(\xi\cdot y))}\dd\mu_q(y)$
reduces to
\begin{equation}
\int_{\sph^q}\abs{v_n(x\cdot y)-v_n(\xi\cdot y)}\dd \mu_q(y)=\sum_{j=1}^2\int_{\sph^q}\mathbf{1}_{U_j}(y)\abs{v_n(x\cdot y)-v_n(\xi\cdot y)}\dd \mu_q(y)\label{eq:split}
\end{equation}
where we split $\sph^q$ into the two sets $U_1\defeq\setcond{y\in\sph^q}{\sin(d(x,y))\leq \sin(d(\xi,y))}$ and $U_2\defeq \sph^q\setminus U_1$, 
see \cref{fig:sphere-decomposition} for an illustration.

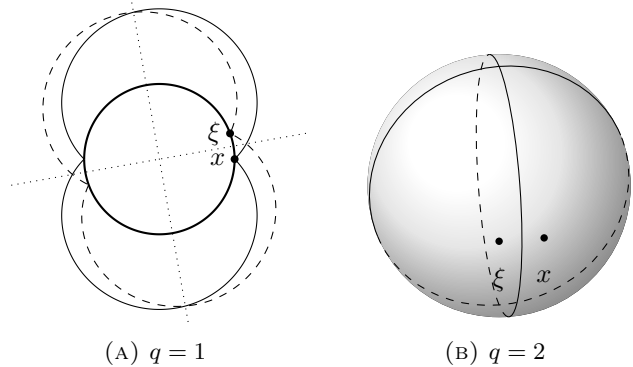
\begin{figure}[ht]
\begin{center}
\subcaptionbox{$q=1$}{
\begin{tikzpicture}
\draw[domain=-pi:pi,samples=500] plot ({deg(\x)}:{1+sin(abs(\x r))});
\draw[domain=-pi:pi,samples=500,dashed] plot ({deg(\x +pi/9)}:{1+sin(abs(\x r))});
\draw[domain=-pi:pi,samples=500,thick] plot ({deg(\x)}:{1});
\fill ({deg(pi/9)}:{1}) circle(1.5pt) node[left]{$\xi$};
\fill ({0}:{1}) circle(1.5pt) node[left]{$x$};
\draw[dotted] ({deg(pi/18)}:{-2})--({deg(pi/18)}:{2});
\draw[dotted] ({deg(10*pi/18)}:{-2.2})--({deg(10*pi/18)}:{2.2});
\end{tikzpicture}
}\quad
\subcaptionbox{$q=2$}{
\begin{tikzpicture}[scale=0.7,every node/.style={minimum size=1cm}]

\def\R{2.5} 
\def\angEl{25} 
\def\angAz{-100} 
\def\angPhiOne{10} 
\def\angPhiTwo{100} 
\def\angBeta{30} 

\pgfmathsetmacro\H{\R*cos(\angEl)} 
\LatitudePlane[xzplane]{\angEl}{\angAz}
\LongitudePlane[pzplane]{\angEl}{\angPhiOne}
\LongitudePlane[qzplane]{\angEl}{\angPhiTwo}
\LatitudePlane[equator]{\angEl}{0}
\fill[ball color=white!10] (0,0) circle (\R); 
\coordinate (O) at (0,0);
\DrawLongitudeCircle[\R]{\angPhiOne} 
\DrawLongitudeCircle[\R]{\angPhiTwo} 
\LatitudeCircle[\R]{0} 
\path[xzplane] (-90:\R) coordinate (P);
\fill (P) circle(2pt) node[below]{$\xi$};
\path[xzplane] (-70:\R) coordinate (Q);
\fill (Q) circle(2pt) node[below]{$x$};
\end{tikzpicture}
}
\end{center}
\caption{For non-antipodal points $x$ and $\xi$ on $\sph^q$, the sets $U_1$ and $U_2$ each take two opposite quarters of the sphere.
In the left panel the dashed thin line shows the sine of the geodesic distance to $\xi$ and the solid thin one depicts sine of the geodesic distance to $x$.
In the right panel, the dashed and solid thin lines show the boundary of $U_1$.}\label{fig:sphere-decomposition}
\end{figure}

Denoting by $y(\tau,\varphi)=\cos(\tau)x+\sin(\tau)\tilde{y}$, $\tau\in[0,\piup]$, $\tilde{y}\in\setcond{z\in\sph^q}{x\cdot z=0}\cong \sph^{q-1}$ the polar coordinates of $y\in\sph^q$ with respect to $x$ as the north pole we obtain
\begin{align*}
&\norel \int_{\sph^q}\mathbf{1}_{U_1}(y)\abs{v_n(x\cdot y)-v_n(\xi\cdot y)}\dd \mu_q(y)\\
&=\int_{\sph^{q-1}}\int_0^\piup \mathbf{1}_{U_1}(y(\tau,\tilde{y}))\abs{v_n(x\cdot y(\tau,\tilde{y}))-v_n(\xi\cdot y(\tau,\tilde{y}))}\sin(\tau)^{q-1}\dd \tau\,\dd\mu_{q-1}(\tilde{y})\\
&=\int_{\sph^{q-1}}\int_0^\piup \mathbf{1}_{U_1}(y(\tau,\tilde{y}))\int_{[d(x,y(\tau,\tilde{y})),d(\xi,y(\tau,\tilde{y}))]}\hspace{-1.2cm}\abs{v_n^\prime(\cos(t))\sin(t)}\,\dd t\ \sin(\tau)^{q-1}\dd \tau\,\dd\mu_{q-1}(\tilde{y}).
\end{align*}
As the sine function is concave on $[d(x,y(\tau,\tilde{y})),d(\xi,y(\tau,\tilde{y}))]\subseteq[0,\piup]$, it attains its minimum on the boundary, which is $d(x,y(\tau,\tilde{y}))=\tau$ in the case of $y(\tau,\tilde{y})\in U_1$, and thus $\sin(\tau)\leq \sin(t)$ for all $t\in [d(x,y(\tau,\tilde{y})),d(\xi,y(\tau,\tilde{y}))]$.
This leads to the upper bound
\begin{align*}
&\int_{\sph^{q-1}}\int_0^\piup\int_{[d(x,y(\tau,\tilde{y})),d(\xi,y(\tau,\tilde{y}))]}\abs{v_n^\prime(\cos(t))\sin(t)^q}\dd t \dd \tau\dd\mu_{q-1}(\tilde{y})\\
&\leq \int_{\sph^{q-1}}\int_0^\piup\int_0^{d(x,\xi)}\abs{v_n^\prime(\cos(t+\tau))\sin(t+\tau)^q}\dd t \dd \tau\dd\mu_{q-1}(\tilde{y})\\
&\leq \omega_{q-1}\int_0^\piup\int_0^{d(x,\xi)}\abs{v_n^\prime(\cos(t+\tau))\sin(t+\tau)^q}\dd t \dd \tau\\
&= \omega_{q-1} \int_0^{d(x,\xi)}\int_0^\piup\abs{v_n^\prime(\cos(t+\tau))\sin(t+\tau)^q} \dd \tau\dd t.\\
\intertext{Utilizing the periodicity of $(v_n^\prime\circ \cos)\sin^q$ and \cref{thm:dvlp}, we obtain}
&= \omega_{q-1} \int_0^{d(x,\xi)}\int_t^{\piup+t}\abs{v_n^\prime(\cos(\tau))\sin(\tau)^q}\dd \tau\dd t\\
&\leq \omega_{q-1} \cdot d(x,\xi) \cdot 2 \int_0^{\piup}\abs{v_n^\prime(\cos(\tau))\sin(\tau)^q}\dd \tau\\
&\leq 2d(x,\xi)C_q(n+q^2),
\end{align*}

The same manipulations can be applied to the second summand in \eqref{eq:split} but with $\xi$ as the north pole.
We obtain 
\begin{align*}
&\norel\int_{\sph^q}\abs{v_n(x\cdot y)-v_n(\xi\cdot y)}\dd \mu_q(y)\\
&=\sum_{j=1}^2\int_{\sph^q}\mathbf{1}_{U_j}(y)\abs{v_n(x\cdot y)-v_n(\xi\cdot y)}\dd \mu_q(y)\\
&\leq 4\mnorm{\prt}C_q(n+q^2)
\end{align*}
This means that
\begin{equation*}
\norel\mnorm{T_{\smpl,\prt,n}(f)}_{\mu_q,\infty}\leq \mnorm{f}_{\mu_q,\infty}4C_q(n+q^2)\mnorm{\prt}
\end{equation*}
which finishes the proof.
\end{proof}

\begin{Bem}
A modification of the technique used for the case $p=1$ in the second step of the preceding proof can also be applied to the case $p=\infty$.
In contrast to the above strategy we would generate an additional summand of order $(n\mnorm{\prt})^{q+1}$.
\end{Bem}

Now we are ready to prove \cref{thm:deterministic-mz}.

\begin{proof}[Proof of \cref{thm:deterministic-mz}]
We first show that 
\begin{equation}
\abs{\mnorm{P}_{\mu_q,p}-\mnorm{P}_{\mu_q(\smpl,\prt),p}}\leq \mnorm{T_{\smpl,\prt,n}(P)}_{\mu_q,p}.\label{eq:NormEst}
\end{equation}
holds for all $1\leq p\leq\infty$ and every $P\in\Pi_n^q$.
For $1\leq p<\infty$, the triangle inequality, the reproducing property of $v_n$, and Hölder's inequality give
\begin{align*}
&\norel\abs{\mnorm{P}_{\mu_q,p}-\mnorm{P}_{\mu_q(\smpl,\prt),p}}\\
&=\abs{\lr{\sum_{\xi\in\smpl}\int_{\Z_\xi}\abs{P(x)}^p\dd\mu_q(x)}^{\frac{1}{p}}- \lr{\sum_{\xi\in\smpl}\int_{\Z_\xi}\abs{P(\xi)}^p\dd\mu_q(x)}^{\frac{1}{p}}}\nonumber\\
&\leq\lr{\sum_{\xi\in\smpl}\int_{\Z_\xi}\abs{P(x)-P(\xi)}^p\dd\mu_q(x)}^{\frac{1}{p}}\\
&=\lr{\sum_{\xi\in\smpl}\int_{\Z_\xi}\abs{\int_{\sph^q}v_n(x\cdot y)P(y)\,\dd\mu_q(y)-\int_{\sph^q}v_n(\xi\cdot y)P(y)\,\dd\mu_q(y)}^p\dd\mu_q(x)}^{\frac{1}{p}}\\
&=\lr{\sum_{\xi\in\smpl}\int_{\Z_\xi}\abs{\int_{\sph^q}(v_n(x\cdot y)-v_n(\xi\cdot y))P(y)\dd\mu_q(y)}^p\dd\mu_q(x)}^{\frac{1}{p}}\\
&=\lr{\int_{\sph^q}\abs{\int_{\sph^q}\sum_{\xi\in\smpl}\mathbf{1}_{\Z_\xi}(x)(v_n(x\cdot y)-v_n(\xi\cdot y))P(y)\,\dd\mu_q(y)}^p\dd\mu_q(x)}^{\frac{1}{p}}\\
&=\mnorm{T_{\smpl,\prt,n}(P)}_{\mu_q,p}.
\end{align*}
If $p=\infty$, the same arguments yield
\begin{align*}
&\norel\abs{\mnorm{P}_{\mu_q,\infty}-\mnorm{P}_{\mu_q(\smpl,\prt),\infty}}\\
&\leq\sup_{\xi\in\smpl}\mu_q\text{-}\esssup_{x\in\Z_\xi}\abs{P(x)-P(\xi)}\\
&\leq\sup_{\xi\in\smpl}\mu_q\text{-}\esssup_{x\in\Z_\xi}\abs{\int_{\sph^q}(v_n(x\cdot y)-v_n(\xi\cdot y))P(y)\dd\mu_q(y)}\\
&=\mu_q\text{-}\esssup_{x\in\sph^q}\abs{\int_{\sph^q}\sum_{\xi\in\smpl}\mathbf{1}_{\Z_\xi}(x)(v_n(x\cdot y)-v_n(\xi\cdot y))P(y)\dd\mu_q(y)}\\
&=\mnorm{T_{\smpl,\prt,n}(P)}_{\mu_q,1}
\end{align*}
which proves \eqref{eq:NormEst}.\\[\baselineskip] 
In order to show that the linear operator $T_{\smpl,\prt,n}:L_p(\sph^q,\mu_q)\to L_p(\sph^q,\mu_q)$ is bounded for every $p\in [1,\infty]$ with operator norm less or equal to $\eta$ we first note that this follows for $p=1$ and $p=\infty$ from \cref{thm:bounded-operator-ell-1,thm:bounded-operator-ell-infty} and $5C_q(n+q^2)\mnorm{\prt}\leq \eta$.
(Note that $\frac{2^{q+3}}{q\Gamma(\frac{q}{2})\Gamma(\frac{q}{2}+1)}+4 C_q<5C_q$ for $q\in\NN$.)
For $1<p<\infty$, the statement follows by the Riesz--Thorin interpolation theorem.

\emph{Step 3: Conclude the assertion.}
From steps 1 and 2, we have
\begin{equation*}
\abs{\mnorm{f}_{\mu_q,p}-\mnorm{f}_{\mu_q(\smpl,\prt),p}}\leq \eta\mnorm{f}_{\mu_q,p}    
\end{equation*}
for all $f\in\Pi_n^q$ whenever $5C_q(n+q^2)\mnorm{\prt}\leq \eta$.
This is equivalent to the assertion.
\end{proof}

The condition $5C_q(n+q^2)\mnorm{\prt}\leq \eta$ appearing in \cref{thm:deterministic-mz} gives a lower bound on the number $N$ of samples through volumetric arguments of the partition.
Namely, if $\mnorm{\prt}\leq\frac{\eta}{5C_q(n+q^2)}\eqdef r$, then $\mu(\Z)\leq \omega_{q-1}\int_0^{r} \sin(t)^{q-1}\dd t$ for each $\Z\in\prt$, and as $\prt$ is a partition of $\sph^q$, the cardinality of $\prt$ is 
\begin{equation*}
N\geq \frac{\omega_q}{\lr{\omega_{q-1}\int_0^{r} \sin(t)^{q-1}\dd t}} \gtrsim_q r^{-q}.
\end{equation*}

As a corollary, we obtain a seemingly partition-free variant of \cref{thm:deterministic-mz} with the upper bound on the partition norm $\mnorm{\prt}$ replaced by an upper bound on the mesh norm $\delta_\smpl$.
It relies on the construction of a partition $\prt$ from the sample set $\smpl$ such that the partition norm and the mesh norm satisfy a two-sided inequality, and hiding the partition in the weights of the discretized norm.
\begin{Kor}
Let $n,q\in\NN$ and $\eta\in \mathopen]0,1\mathclose[$.
Let further $\smpl\subset\sph^q$ be a finite set satisfying
\begin{equation*}
40C_qq\sqrt{2q(q+1)}(n+q^2)\delta_\smpl\leq \eta
\end{equation*}
with $C_q$ as in \cref{eq:constant}.
Then there exist non-negative numbers $a_\xi$, $\xi\in\smpl$, such that
\begin{equation*}
(1-\eta)\mnorm{f}_{\mu_q, p}\leq\lr{\sum_{\xi\in\smpl}a_\xi\abs{f(\xi)}^p}^{\frac{1}{p}}\leq (1+\eta)\mnorm{f}_{\mu_q,p}.
\end{equation*}
for all $p\in[1,\infty]$ and $f\in\Pi_n^q$.
\end{Kor}
\begin{proof}
Let $\smpl=\setn{\xi_1,\ldots,\xi_N}\subseteq\sph^q$.
Then \cite[Proposition~3.2]{MhaskarNaWa2001} gives a partition $\prt=\setn{\Z_1,\ldots,\Z_M}$ of $\sph^q$ for some $M\leq N$ such that there exists an $M$-element subset $\smpl_0$ of $\smpl$ for which the pair $(\smpl_0,\prt)$ is compatible, and the inequality
\begin{equation*}
\delta_\smpl\leq\mnorm{\prt}\leq 8q\sqrt{2q(q+1)}\delta_\smpl
\end{equation*}
is satisfied.
We plug this into $40C_qq\sqrt{2q(q+1)}(n+q^2)\delta_\smpl\leq \eta$ to obtain $5C_q\mnorm{\prt}(n+q^2)\leq \eta$.
Set
\begin{equation*}
a_\xi=\begin{cases}\mu_q(\Z_\xi)&\text{if }\xi\in\smpl_0,\\0&\text{else,}\end{cases} 
\end{equation*}
and apply \cref{thm:deterministic-mz}.
\end{proof}

Via equal-area partitions of the sphere, we can also get an equal-weight version of \cref{thm:deterministic-mz}.
\begin{Kor}
Let $n,q\in\NN$ and $\eta\in \mathopen]0,1\mathclose[$.
For $\alpha_q\defeq 8\lr{\frac{\omega_qq}{\omega_{q-1}}}^{\frac{1}{q}}$ and
\begin{equation*}
N\geq \lr{\frac{5C_q(n+q^2)\alpha_q}{\eta}}^q,
\end{equation*}
there exists a finite subset $\smpl=\setn{\xi_1,\ldots,\xi_N}$ of $\sph^q$ with
\begin{equation*}
(1-\eta)\mnorm{f}_{\mu_q, p}\leq \lr{\frac{\omega_q}{N}\sum_{j=1}^N\abs{f(\xi_j)}^p}^{\frac{1}{p}}\leq (1+\eta)\mnorm{f}_{\mu_q,p}
\end{equation*}
for all $p\in[1,\infty]$ and $f\in\Pi_n^q$.
\end{Kor}
\begin{proof}
Using , there exists  exists a number $\alpha_q$ (which may only depend on $q$) and a partition $\prt=\setn{\Z_1,\ldots,\Z_N}$ of $\sph^q$ such that $\mu_q(\Z_j)=\frac{\omega_q}{N}$ and $\mnorm{\prt}\leq \alpha_q N^{-\frac{1}{q}}$.
According to \cite[\rus{Теорема}~2.3]{BilykLa2017}, one may have $\alpha=8\lr{\frac{\omega_qq}{\omega_{q-1}}}^{\frac{1}{q}}$.
Plugging this into $N\geq \lr{\frac{5C_q(n+q^2)\alpha_q}{\eta}}^q$, we get $5C_q(n+q^2)\mnorm{\prt}\leq \eta$.
For assembling the set $\smpl$, just take one interior point of each $\Z\in\prt$.
It remains to apply \cref{thm:deterministic-mz}.
\end{proof}


\section{How good are random points?}\label{chap:final}
In the previous section we have seen that the performance of sample points (regarding the disturbance parameter $\eta>0$) improves with smaller partition norm $\mnorm{\prt}$ of the corresponding partition, or, differently, with a smaller mesh norm $\delta_{\smpl}$.
In this section we want to look at mesh norms for random points drawn independently and identically distributed from the sphere $\sph^q$.
We show that, when drawing enough points, we obtain a good Marcinkiewicz--Zygmund inequality for all $1\leq p \leq \infty$ simultaneously.
Compared to the consideration in \cref{sec:random_first} we obtain a worse scaling of the number of points with respect to $\eta$.
Note, that in \cref{sec:random_first} we considered only $p=2$ and observed a scaling of $t^{-2}$ independently of the dimension $q$ and explicit constants on the price of an additional logarithm in the dimension.

The main result in this section (\cref{thm:leopardi-mz}) utilizes insights about a classical problem of probability theory: the coupon collector problem.
At each time step, the eponymous coupon collector receives a coupon, chosen at random among $M$ different types.
Unsurprisingly, the more coupons the collector receives, the higher the probability that the collection contains each type at least once.
The time step after which the collector possesses each type at least once can be modeled as a random variable $T_M:(\Omega,\mathcal{F},\PP)\to\RR$, where $(\Omega,\mathcal{F},\PP)$ is some probability space, see \cite[p.~36]{Doerr2020}.
As we do not know the number of draws a priori, we should therefore start with an infinite product of the uniform probability space over the set $\setn{1,\ldots,M}$ of coupon types.
To circumvent this, we raise the probability space to a sufficiently high power $t$, and model the event of not having all $M$ types of coupons after $t$ draws directly as a subset of a \emph{finite} probability space.
\begin{Prop}\label{thm:coupon}
Let $\eps\in \mathopen]0,1\mathclose[$, $M\in\NN$, and $t>M\log\lr{\frac{M}{\eps}}$.
On the finite set $\Omega\defeq \setn{1,\ldots,M}^t$, a probability measure $\PP$ is given by the $t$-fold product measure on the uniform probability measure on $\setn{1,\ldots,M}$.
Set
\begin{equation*}
A_{t,M}\defeq\setcond{(x_1,\ldots,x_t)\in\Omega}{\setn{x_1,\ldots,x_t}=\setn{1,\ldots,M}}.
\end{equation*}
Then $\PP(\Omega\setminus A_{t,M})<\eps$.
\end{Prop}
\begin{proof}
The assumption $t>M\log\lr{\frac{M}{\eps}}$ is equivalent to $M\exp(-\frac{t}{M})<\eps$.
Since $(1+a)^t\leq \eps(ax)$ for all $a,x\in\RR$, we have $M\lr{1-\frac{1}{M}}^t\leq M\exp(-\frac{t}{M})<\eps$.
Taking \cite[p.~36]{Doerr2020} into account, we have $\PP(\Omega\setminus A_{t,M})\leq M\lr{1-\frac{1}{M}}^t<\eps$.
\end{proof}

Now a probabilistic $L_p$-version of the Marcinkiewicz--Zygmund inequality can be given as follows.
\begin{Satz}\label{thm:leopardi-mz}
Let $n,q\in\NN$, $\eta,\eps\in \mathopen]0,1\mathclose[$, and set $\alpha_q\defeq 8\lr{\frac{\omega_qq}{\omega_{q-1}}}^{\frac{1}{q}}$.
Choose $N\in\NN$ large enough such that
\begin{equation}
5C_q\alpha_q\lr{\frac{N}{2\log(\frac{N}{\eps})}}^{-\frac{1}{q}}(n+q^2)<\eta.\label{eq:random-mz}
\end{equation}
Draw points $\xi_1,\ldots,\xi_N\in\sph^q$ independently and identically distributed according to $\sigma_q$.
Then with probability $\geq 1-\frac{1}{N}$ with respect to the product measure $\PP=\sigma_q^{\otimes N}$, there exists weights $w_1,\ldots,w_N>0$ such that $\sum_{j=1}^N w_j=1$ and 
\begin{equation*}
(1-\eta)\mnorm{P}_{\sigma_q,p}\leq \lr{\sum_{j=1}^N w_j \abs{P(\xi_j)}^p}^\frac{1}{p}\leq (1+\eta)\mnorm{P}_{\sigma_q,p}
\end{equation*}
for all $p\in [1,\infty]$ and all $P\in\Pi_n^q$.
\end{Satz}

\begin{proof}
Let $M\defeq\floor{\frac{N}{\log(\frac{N}{\eps})}}$.
From $M\leq \frac{N}{\log(\frac{N}{\eps})}$, we infer $N>M\log(\frac{N}{\eps})>M\log(\frac{M}{\eps})$.
Furthermore, we have $\floor{z}\geq \frac{1}{2}z$ for all $z\in\RR$ with $z>1$.
For $z=\frac{N}{\log(\frac{N}{\eps})}$, we obtain $M=\floor{\frac{N}{\log(\frac{N}{\eps})}}\geq \frac{N}{2\log(\frac{N}{\eps})}$.
Thus \cref{eq:random-mz} implies
\begin{equation*}
5C_q\alpha_qM^{-\frac{1}{q}}(n+q^2)\leq\eta.
\end{equation*}
Using \cite[Theorem~3.1.3]{Leopardi2007} and \cite[\rus{Теорема}~2.3]{BilykLa2017}, there exists a partition $\prt=\setn{\Z_1,\ldots,\Z_M}$ of $\sph^q$ such that $\sigma_q(\Z_j)=\frac{1}{M}$ and $\mnorm{\prt}\leq \alpha_q M^{-\frac{1}{q}}$.
It follows that
\begin{equation*}
5C_q(n+q^2)\mnorm{\prt}<\eta.
\end{equation*}
Thus the conditions of \cref{thm:deterministic-mz} is are met if there is an $M$-element subset $\smpl\subseteq\setn{\xi_1,\ldots,\xi_N}$ such that $(\smpl,\prt)$ is compatible.
For this, we use \cref{thm:coupon}.
It implies that after drawing $N>M\log(\frac{M}{\eps})$ points $\xi_1,\ldots,\xi_N\in\sph^q$ independently and identically distributed according to $\sigma_q$, the probability that each patch $\Z_j\in\prt$ contains a non-zero number $m_j$ of the points $\xi_1,\ldots,\xi_N$ in its interior is $\geq 1-\eps$.
Collect one of the points in each patch in a set $\smpl$, and apply \cref{thm:deterministic-mz}.
This implies that, with probability $\geq 1-\eps$, we have
\begin{equation}
(1-\eta)\mnorm{P}_{\sigma_q,p}\leq \lr{\frac{1}{M}\sum_{\xi\in\smpl} \abs{P(\xi)}^p}^\frac{1}{p}\leq (1+\eta)\mnorm{P}_{\sigma_q,p}\label{eq:leopardi-mz-intermediate}
\end{equation}
for all $p\in [1,\infty]$ and all $P\in\Pi_n^q$.
This result is independent of what point from a given patch $\Z_j$ is put into the set $\smpl$.
In particular, \eqref{eq:leopardi-mz-intermediate} holds true if $\xi\in \Z_j$ is selected such that $\abs{P(\xi)}^p$ is smallest possible or largest possible.
Therefore \eqref{eq:leopardi-mz-intermediate} will also hold if we replace the contribution $\abs{P(\xi)}^p$ from each patch by the average $\frac{1}{m_j}\sum_{\xi\in \Z_j}\abs{P(\xi)}^p$, i.e.,
\begin{equation*}
(1-\eta)\mnorm{P}_{\sigma_q,p}\leq \lr{\frac{1}{M}\sum_{j=1}^M\frac{1}{m_j}\sum_{\xi\in \Z_j} \abs{P(\xi)}^p}^\frac{1}{p}\leq (1+\eta)\mnorm{P}_{\sigma_q,p}
\end{equation*}
or, equivalently,
\begin{equation*}
(1-\eta)\mnorm{P}_{\sigma_q,p}\leq \lr{\sum_{j=1}^M\sum_{\xi\in \Z_j}\frac{1}{Mm_j} \abs{P(\xi)}^p}^\frac{1}{p}\leq (1+\eta)\mnorm{P}_{\sigma_q,p}.
\end{equation*}
Now set $w_j\defeq \frac{1}{Mm_j}>0$ and observe that
\begin{equation*}
\sum_{j=1}^Nw_j=\sum_{j=1}^M \sum_{\xi\in \Z_j}\frac{1}{Mm_j}=\sum_{j=1}^M \frac{m_j}{Mm_j}=1.\qedhere
\end{equation*}
\end{proof}

\textbf{Acknowledgements.} F.F.\ was partly supported by projects AsoftXm ZT-I-PF-4-018 and BRELMMM ZT-I-PF-4-024 of the Helmholtz Imaging Platform (HIP).
The work of T.J.\ was supported by the German Science Foundation (DFG) through grant 1742243256 - TRR 96 and DFG Ul-403/2-1.
Also, the authors want to thank Feng Dai for pointing out reference \cite{BrownDa2005}.

\providecommand{\bysame}{\leavevmode\hbox to3em{\hrulefill}\thinspace}
\providecommand{\MR}{\relax\ifhmode\unskip\space\fi MR }
\providecommand{\MRhref}[2]{%
  \href{http://www.ams.org/mathscinet-getitem?mr=#1}{#2}
}
\providecommand{\href}[2]{#2}

\end{document}